\documentclass[11pt]{article}
\usepackage{amssymb}
\usepackage{amsmath}
\usepackage{amsthm}
\usepackage{latexsym}
\usepackage[english]{babel}

\usepackage{graphicx}
\usepackage{epstopdf}
\usepackage{amsmath}
\usepackage{amsfonts}
\usepackage{amssymb}
\usepackage{amsmath}
\usepackage{amsthm}
\usepackage{latexsym}
\usepackage{graphicx}
\usepackage{epstopdf}
\usepackage{float}
\usepackage{caption}
\usepackage{subfigure}
\usepackage{tikz}
\usetikzlibrary{matrix,arrows}
\newtheorem{theorem}{Theorem}[section]

\newtheorem{corollary}[theorem]{Corollary}
\newtheorem{definition}[theorem]{Definition}

\newtheorem{lemma}[theorem]{Lemma}

\newtheorem{proposition}[theorem]{Proposition}
\newtheorem{remark}[theorem]{Remark}

\begin{document}

\title{The Rabinowitz-Floer homology for a class of semilinear problems and applications}

\author{Ali Maalaoui$^{(1)}$ \& Vittorio Martino$^{(2)}$}
\addtocounter{footnote}{1}
\footnotetext{Department of mathematics and natural sciences, American University of Ras Al Khaimah, PO Box 10021, Ras Al Khaimah, UAE. E-mail address:
{\tt{ali.maalaoui@aurak.ae}}}
\addtocounter{footnote}{1}
\footnotetext{Dipartimento di Matematica, Universit\`a di Bologna,
piazza di Porta S.Donato 5, 40127 Bologna, Italy. E-mail address:
{\tt{vittorio.martino3@unibo.it}}}

\date{}
\maketitle

\vspace{5mm}
\begin{center}
    \emph{Dedicated to Abbas Bahri on his sixtieth birthday.}
\end{center}

\vspace{5mm}

{\noindent\bf Abstract} {\small In this paper, we construct a Rabinowitz-Floer type homology for a class of non-linear problems having a \emph{starshaped} potential; we consider some equivariant cases as well. We give an explicit computation of the homology and we apply it to obtain results of existence and multiplicity of solutions for several model equations.}

\vspace{5mm}

\vspace{5mm}
\tableofcontents

\noindent

%%%%%%%%%%%%%%%%%%%%%%%%%%%%%%%%%%%%%%%%%%%%%%%%%%%%%%%%%%%%%%%%%%%%%%%%%%%%%%%%%%%%%%%%%%%%%%%%%%%%%%%%%%%%%%%%%%%%%%%%%%%%
%%%%%%%%%%%%%%%%%%%%%%%%%%%%%%%%%%%%%%%%%%%%%%%%%%%%%%%%%%%%%%%%%%%%%%%%%%%%%%%%%%%%%%%%%%%%%%%%%%%%%%%%%%%%%%%%%%%%%%%%%%%%
%%%%%%%%%%%%%%%%%%%%%%%%%%%%%%%%%%%%%%%%%%%%%%%%%%%%%%%%%%%%%%%%%%%%%%%%%%%%%%%%%%%%%%%%%%%%%%%%%%%%%%%%%%%%%%%%%%%%%%%%%%%%

\section{Introduction and main results}

\noindent
In this paper we will consider a class of semilinear functionals defined on some Hilbert subspaces as space of variations and under suitable assumptions on the nonlinearity, we will define a Morse-type complex following the idea in the Rabinowitz-Floer approach; if in addition the functionals are $S^1$- or $\mathbb{Z}_2$-equivariant, then we will have further reductions in the homology computation.

\noindent
Thus, let $E$ be a Hilbert space and let $H\subset E$  be a dense subspace compactly embedded in $E$. We consider a linear operator
$$L:H\longrightarrow E$$
invertible and auto-adjoint. Hence $L$ will have a basis of eigenfunctions $\{\varphi_{i}\}_{i\in \mathbb{Z}}$
$$L(\varphi_i)=\lambda_i \varphi_i$$
with the convention that if $\lambda_i>0$ then $i>0$. This allows us to define the unbounded operator $|L|^{\frac{1}{2}}$ in the following way; if
$$u=\sum_{i\in \mathbb{Z}}a_{i}\varphi_{i}$$
then
$$L(u)=\sum_{i\in \mathbb{Z}}\lambda_{i}a_{i}\varphi_{i}$$
and therefore
$$|L|^{\frac{1}{2}}u=\sum _{i\in \mathbb{Z}}|\lambda_{i}|^{\frac{1}{2}}a_{i}\varphi_{i}$$
Now we denote by $\langle \cdot,\cdot \rangle$ the inner product in $E$, and we define then the inner product of $H$ as follows $$\langle u,v \rangle_{H} = \langle |L|^{\frac{1}{2}}u,|L|^{\frac{1}{2}}v \rangle$$
We obtain also the decomposition
$$H=H^{+}\oplus H^{-} $$
where
$$H^{-}=\overline{span\{\varphi_{i},i<0 \}}, \qquad H^{+}=\overline{span\{\varphi_{i},i>0 \}}$$
We will write
$$u=u^{+}+u^{-}, \qquad \forall \; u\in H$$
according to the previous splitting.
We explicitly note that
$$L(u^{+} + u^{-}) = |L|(u^{+} - u^{-})$$
Now we consider the following functional defined on $\mathcal{H}=H\times \mathbb{R}$ by
$$I(u,\lambda)=\frac{1}{2}\langle Lu,u \rangle -\lambda F(u)$$
where $F:E\rightarrow\mathbb{R}$ is a $C^{2}$ function with the following properties :

\begin{itemize}
\item[(F1)]  $|L|^{-1}\nabla F$ is compact
\item[(F2)] the set $S=\{u\in E \; s.t. \;F(u)=0\}$ bounds a strictly starshaped bounded domain in $E$ \\
\end{itemize}

\noindent
In the following, if $(F2)$ holds, we will simply say that $S$ is a bounded strictly starshaped surface and $F$ is a starshaped potential. Our main result is the following
\begin{theorem}
If $F$ satisfies the hypotheses $(F1)$ and $(F2)$ then the Rabinowitz-Floer homology $H_{*}(I)$ is well defined. Moreover
$$H_{*}(I)=0$$
If in addition $I$ is $S^{1}$-equivariant or $\mathbb{Z}_2$- equivariant respectively, then we have
$$H_{*}^{S^{1}}(I)=H_{*}(\mathbb{C}P^{\infty})=\left\{
\begin{array}{ll}
\mathbb{Z}_2 & \mbox{if $*$ is even}\\
\\
0 & \mbox{otherwise}\\
\end{array}
\right.
$$
and
$$H_{*}^{\mathbb{Z}_2}(I)=\mathbb{Z}_2$$
respectively.
\end{theorem}

\begin{remark}
From the proof of the result (see Section 6 also) we will show that in particular the homology we compute depends strongly on the set $S$ in the hypothesis $(F2)$, rather than the defining function $F$; therefore we could use the notation $H_{*}(S)$ as well.
\end{remark}

\vspace{5mm}

\noindent
{\bf Acknowledgement}
We explicitly thank Professor Takeshi Isobe for the fruitful discussions and the useful suggestions on the subject. The authors aknowledge the financial support of the Seed Grant of AURAK, No.: AAS/001/15, \emph{Extensions of the Rabinowitz-Floer Homology and Applications to more General PDEs}\\

%%%%%%%%%%%%%%%%%%%%%%%%%%%%%%%%%%%%%%%%%%%%%%%%%%%%%%%%%%%%%%%%%%%%%%%%%%%%%%%%%%%%%%%%%%%%%%%%%%%%%%%%%%%%%%%%%%%%%%%%%%%%
%%%%%%%%%%%%%%%%%%%%%%%%%%%%%%%%%%%%%%%%%%%%%%%%%%%%%%%%%%%%%%%%%%%%%%%%%%%%%%%%%%%%%%%%%%%%%%%%%%%%%%%%%%%%%%%%%%%%%%%%%%%%
%%%%%%%%%%%%%%%%%%%%%%%%%%%%%%%%%%%%%%%%%%%%%%%%%%%%%%%%%%%%%%%%%%%%%%%%%%%%%%%%%%%%%%%%%%%%%%%%%%%%%%%%%%%%%%%%%%%%%%%%%%%%

\section{Relative index and Moduli space of trajectories}

\noindent
First of all, we note explicitly that critical points of $I$ satisfy the following equations:
\begin{equation}\label{cpeqs}
\left\{
\begin{array}{ll}
L u = \lambda \nabla F(u)\\
\\
F(u)=0\\
\end{array}
\right.
\end{equation}

\noindent
Moreover if we compute the Hessian of $I$ at a critical point $(\lambda,u)$, we get
$$Hess(I)(u,\lambda)=\left(\begin{array}{lc}
|L|^{-1}L -\lambda |L|^{-1} \nabla^{2} F(u) & -|L|^{-1}\nabla F(u)\\
\\
-|L|^{-1}\nabla F(u) & 0
\end{array}
\right)
$$
Hence the index and co-index of the critical point are infinite. So we need to introduce an alternative way of grading (as in \cite{AB2} for instance).
\begin{definition}
Consider two closed subspaces $V$ and $W$ of a Hilbert space $E$. We say that $V$ is a compact perturbation of $W$ if $P_{V}-P_{W}$ is a compact operator.
\end{definition}

\noindent
$P_{V}$ in the previous definition denote the orthogonal projection on $V$. Now, if $V$ is a compact perturbation of $W$, we can define the relative dimension as
$$dim(V,W)=dim(V\cap W^{\perp})-dim(V^{\perp}\cap W).$$
One can check that it is well defined and finite. Now if we have three subspaces $V$, $W$ and $U$ such that $V$ and $W$ are compact perturbations of $U$. Then $V$ is also a compact perturbation of $W$ and
$$dim(V,W)=dim(V,U)+dim(U,W).$$
Using this concept of relative dimension we can define a relative index as our grading.
\begin{definition}
We denote by $V^{-}(u,\lambda)$ the closure of the span of the eigenfunction of the Hessian of $I$ at a critical point $(u,\lambda)$, corresponding to negative eigenvalues.\\
The relative index is defined as
$$i_{rel}(u,\lambda)=dim(V^{-}(u,\lambda),H^{-}\times \mathbb{R})$$
\end{definition}
\begin{lemma}
If $I$ is Morse and $(F1)$ holds then the relative index is well defined for critical points of $I$.
\end{lemma}
\begin{proof}
Let $\Gamma=H^{-}\times \mathbb{R}$, and $(u,\lambda)$ a critical point of $I$. The operator
$$v\mapsto Lv-\lambda \nabla ^{2} F(u)v$$
has discrete spectrum since $|L|^{-1}$ is a compact operator. Then $V^{-}(u,\lambda)$ is well defined and it is a compact perturbation of $\Gamma$. This follows from the fact that
$$\left( \begin{array}{cc}
|L|^{-1}L & 0 \\
0 & 1
\end{array}
\right) - \left( \begin{array}{cc}
|L|^{-1}L - \lambda |L|^{-1}\nabla^{2}F(u) & |L|^{-1}\nabla F(u) \\
|L|^{-1}\nabla F(u) & 0
\end{array}
\right),$$
is a compact operator.
\end{proof}

\noindent
We define now the moduli space of $\mathcal{H}$-gradient trajectories. Let us consider the following differential system:
\begin{equation}\label{GF}
\left\{\begin{array}{ll}
\displaystyle\frac{\partial u}{\partial t}=u^{-}-u^{+}+\lambda |L|^{-1}\nabla F(u)\\
\\
\displaystyle\frac{\partial \lambda}{\partial t}=F(u)
\end{array}
\right.
\end{equation}

\noindent
This system is in fact the descending gradient flow of our  functional in $\mathcal{H}=H\times \mathbb{R}$ and since the right hand side is smooth, then we have local existence of the flow.\\
We see that one is tempted to use a sort of ``$L^{2}$ gradient'' flow, to get a heat flow equation, but in this case the problem is ill-posed since the spectrum of $L$ is unbounded from below.\\
The type of gradient flow we defined before was used by A.Bahri in \cite{B}, where it appears that it has better properties than the classical heat flow and it is moreover positivity preserving.\\
Now, given two critical points $z_{0}=(u_{0},\lambda_{0})$ and $z_{1}=(u_{1},\lambda_{1})$ such that
$$I(z_{i})\in [a,b], \qquad \mbox{for} \quad i=0,1$$
we define the space of connecting orbits from $z_{0}$ to $z_{1}$ by
$$\mathbb{M}^{a,b}(z_{0},z_{1})=\left\{z= (u,\lambda)\in C^{1}(\mathbb{R},\mathcal{H})\; \text{satisfying (\ref{GF}) with } z(-\infty)=z_{0}\; , z(+\infty)=z_{1}\right\}$$
where we have denoted by
$$z(-\infty):=\lim_{t\rightarrow -\infty} z(t), \qquad z(+\infty):=\lim_{t\rightarrow +\infty} z(t)$$
The moduli space of trajectories is then defined by $$\mathcal{M}^{a,b}(z_{0},z_{1})=\mathbb{M}^{a,b}(z_{0},z_{1})/\mathbb{R}.$$
\begin{proposition}
Assume that $i_{rel}(z_{0})> i_{rel}(z_{1})$, then if $I$ is Morse-Smale,
$$dim\Big(\mathcal{M}^{a,b}(z_{0},z_{1})\Big)=i_{rel}(z_{0})-i_{rel}(z_{1})-1$$
\end{proposition}
\begin{proof}
We first note that $\mathbb{M}^{a,b}(z_{0},z_{1})=\mathcal{F}^{-1}(0)$ where
$$\mathcal{F}:C^{1}(\mathbb{R},\mathcal{H})\mapsto \mathcal{Q}^{0}=C^{0}(\mathbb{R},\mathcal{H})$$
is defined by
$$\mathcal{F}(z)=\frac{dz}{dt}+\nabla I(z)$$
We will use the implicit function theorem to prove our result: we need to show that the linearized operator of $\mathcal{F}$ is Fredholm and onto. The linearized operator corresponds to
$$\partial \mathcal{F}(z)=\frac{d}{dt}+Hess(I(z))$$
and this is a linear differential equation in the Banach space $\mathcal{H}$ (see \cite{AB1}). In order to show that it is Fredholm, we first notice that
$$Hess(I(z))=\left( \begin{array}{cc}
|L|^{-1}L & 0\\
0 & 1
\end{array}
\right) +\left( \begin{array}{cc}
-\lambda |L|^{-1}\nabla^{2}F(u) & -|L|^{-1}\nabla F(u)\\
-|L|^{-1}\nabla F(u)& -1
\end{array}
\right).
$$
Now, the operator
$$\left( \begin{array}{cc}
|L|^{-1}L & 0\\
0 & 1
\end{array}
\right)$$
is time independent and hyperbolic, while the operator
$$\left( \begin{array}{cc}
-\lambda |L|^{-1}\nabla^{2} F(u) & -|L|^{-1}\nabla F(u)\\
-|L|^{-1}\nabla F(u) & -1
\end{array}
\right) $$
is compact. Hence we have that $\partial\mathcal{F}$ is a Fredholm operator with index $$ind(\partial\mathcal{F}(z))=dim(V^{-}(\mathcal{F}(z_{0}),V^{-}(\mathcal{F}(z_{1}))$$
$$=dim(V^{-}(\mathcal{F}(z_{0}),\Gamma)+dim(\Gamma,V^{-}(\mathcal{F}(z_{1}))$$
$$=i_{rel}(z_{0})-i_{rel}(z_{1}).$$
Moreover, from \cite{AB1}, we have also that $\partial\mathcal{F}(z)$ is onto if and only if the intersection is transverse.\\
To finish the proof now, it is enough to notice that the action of $\mathbb{R}$ is free and hence we can mod out by that action to get the desired result.
\end{proof}

%%%%%%%%%%%%%%%%%%%%%%%%%%%%%%%%%%%%%%%%%%%%%%%%%%%%%%%%%%%%%%%%%%%%%%%%%%%%%%%%%%%%%%%%%%%%%%%%%%%%%%%%%%%%%%%%%%%%%%%%%%%%
%%%%%%%%%%%%%%%%%%%%%%%%%%%%%%%%%%%%%%%%%%%%%%%%%%%%%%%%%%%%%%%%%%%%%%%%%%%%%%%%%%%%%%%%%%%%%%%%%%%%%%%%%%%%%%%%%%%%%%%%%%%%
%%%%%%%%%%%%%%%%%%%%%%%%%%%%%%%%%%%%%%%%%%%%%%%%%%%%%%%%%%%%%%%%%%%%%%%%%%%%%%%%%%%%%%%%%%%%%%%%%%%%%%%%%%%%%%%%%%%%%%%%%%%%

\section{Compactness}

\subsection{Palais-Smale condition and compactification of the moduli spaces}
We recall that a functional $I$ is said to satisfies the Palais-Smale condition (PS), at the level $c$, if every sequence $(z_{k})$ such that
$$I(z_{k})\longrightarrow c$$
and
$$\nabla I(z_{k})\longrightarrow 0$$
as $k\rightarrow\infty$, has a convergent subsequence. We will say that $I$ satisfies (PS) if the previous condition is satisfied for all $c\in \mathbb{R}$.\\

%\textbf{ is it true that if $F(u_{n})\to 0$ then $d(u_{n},S)\to 0$? or should i assume it? it seems right }

\begin{proposition}
Under the assumption $(F1)$ and $(F2)$, $I$ satisfies the (PS) condition.
\end{proposition}
\begin{proof}

Let $z_{k}=(u_k,\lambda_k)$ be a (PS) sequence at the level $c$, that is
$$\left\{
\begin{array}{ll}
Lu_{k}-\lambda_{k}\nabla F(u_{k})=\varepsilon_{k}\\
\\
F(u_{k})=\varepsilon_{k}
\end{array}
\right.
$$
and
$$I(u_{k},\lambda_{k})=c+\varepsilon_{k}$$
Thus if we consider
$$\langle Lu_{k}-\lambda_{k}\nabla F(u_{k}),u_{k} \rangle -2I(u_{k},\lambda_{k})$$
we get
$$\lambda_{k} \langle \nabla F(u_{k}),u_{k} \rangle =2c+\varepsilon_{k}(\|u_{k}\|_{H}+2|\lambda_{k}|)$$
Since, $F(u_{k})\to 0$ and $S$ is strictly starshaped, there exists $a>0$ such that
$$\langle \nabla F(u_{k}),u_{k}\rangle >a.$$ Thus
\begin{equation}\label{lam}
|\lambda_{k}|=C+\varepsilon_{k}(\|u_{k}\|_{H}+|\lambda_{k}|)
\end{equation}
and we have
$$\langle Lu_{k}-\lambda_{k}\nabla F(u_{k}),u_{k}^{+} \rangle =\|u_{k}^{+}\|_{H}^{2}-\lambda_{k} \langle \nabla F(u_{k}),u_{k}^{+} \rangle$$
and
$$\langle Lu_{k}-\lambda_{k}\nabla F(u_{k}),u_{k}^{-} \rangle=-\|u_{k}^{-}\|_{H}^{2}-\lambda_{k} \langle \nabla F(u_{k}),u_{k}^{-} \rangle$$
Hence,
$$\|u_{k}\|_{H}^{2}=\lambda_{k} \langle \nabla F(u_{k}),u_{k}^{+}-u_{k}^{-} \rangle+\varepsilon_{k}\|u_{k}\|_{H}$$
therefore
$$\|u_{k}\|_{H}^{2}\leq C|\lambda_{k}| \|\nabla F(u_{k})\| |u_{k}|+\varepsilon_{k}\|u_{k}\|_{H}$$
Now, since $S$ is a bounded hypersurface, we have that $u_{k}$ is bounded in $E$ thus
\begin{equation}\label{u}
\|u_{k}\|_{H}\leq C|\lambda_{k}|+\varepsilon_{k}
\end{equation}
By adding (\ref{u}) and (\ref{lam}), one has
$$\|u_{k}\|_{H}+|\lambda_{k}|\leq C+C\varepsilon_{k}(\|u_{k}\|_{H}+|\lambda_{k}|)+\varepsilon_{k}$$
Thus both $u_{k}$ and $\lambda_{k}$ are bounded. Hence we can extract a weakly convergent subsequence of $u_{k}$ and a convergent subsequence of $\lambda_{k}$. That is $\lambda_{k}\to \lambda$ and $u_{k}\rightharpoonup u$.
Now notice since $|L|^{-1} \nabla F$ is compact we have that
$$\|u_{k}^{+}\|_{H}^{2}\to \lambda \langle \nabla F(u),u^{+} \rangle$$
similarly
$$\|u_{k}^{-}\|_{H}^{2}\to -\lambda \langle \nabla F(u),u^{-} \rangle$$
But again,
$$\langle Lu_{k}-\lambda_{k}\nabla F(u_{k}),u^{+} \rangle \to \|u^{+}\|_{H}^{2}-\lambda \langle \nabla F(u),u^{+} \rangle$$
and
$$\langle Lu_{k}-\lambda_{k}\nabla F(u_{k}),u^{-}\rangle \to -\|u^{-}\|_{H}^{2}-\lambda \langle \nabla F(u),u^{-} \rangle$$
Combining both we get
$$\|u\|_{H}^{2}=\lambda \langle \nabla F(u),u^{+}-u^{-} \rangle$$
Hence we have the convergence in norm of $u_{k}$ and thus (PS) holds.

%Let $z_{k}=(u_k,\lambda_k)$ be a (PS) sequence at the level $c$, that is
%$$\left\{
%\begin{array}{ll}
%Lu_{k}-\lambda_{k}\nabla F(u_{k})=o(1)\\
%\\
%F(u_{k})=o(1)
%\end{array}
%\right.
%$$
%and
%$$I(u_{k},\lambda_{k})=c+o(1)$$
%Thus we have that
%$$\lambda_{k}(\langle \nabla F(u_{k}),u_{k}\rangle - F(u_{k}))=o(\|u_{k}\|)$$
%hence
%$$\lambda_{k}\langle \nabla F(u_{k}),u_{k}>=o(\|z_{k}\|)$$
%In the other hand we have that
%$$\|u_{k}^{+}\|^{2}-\lambda_{k} \langle \nabla F(u_{k}),u_{k}^{+}\rangle = o(\|u_{k}^{+}\|)$$
%Therefore
%$$\|u_{k}^{+}\|^{2}=o(\|z_{k}\|)+C_{1}+C_{2}F(u_{k})$$
%A similar inequality holds for $\|u_{k}^{-}\|$ and this gives us boundedness of $z_{k}$ in $\mathcal{H}$. Thus we can extract a subsequence convergent strongly in $E$. Going back to the equation, we have
%$$\|u_{k}^{+}\|^{2}-\lambda_{k}\langle \nabla F(u_{k}),u_{k}\rangle=o(1)$$
%thus passing to the limit, we see that we have strong convergence in norm, hence we have convergence of the subsequence in $\mathcal{H}$.
\end{proof}

\noindent
In fact following the previous proof, we have proved the following
\begin{lemma}\label{lemmaestimate}
If
$$\|\nabla I(u,\lambda)\|\leq \varepsilon$$
and
$$|I(u,\lambda)|\leq M$$
for some $\varepsilon, M >0$, then there exists $C=C(M)>0$ such that
$$|\lambda|\leq C$$
and
$$\|u\|_{H}\leq C.$$
\end{lemma}

\noindent
\begin{proposition}
Under hypotheses $(F1)$ and $(F2)$, the flow lines between critical points (or more precisely the moduli space) are uniformly bounded.
\end{proposition}

\begin{proof}
Let
$$z(t)=(u(t),\lambda(t))$$
a flow line of the negative gradient of $I$, that is $z$ is a solution of (\ref{GF}). First of all we have
$$\int_{-\infty}^{+\infty} \|z'\|^{2}dt\leq b-a$$
Our first objective is to show the boundedness of the multiplier $\lambda$. So we consider the following function:

\begin{equation}\label{tau}
\tau(s)=\inf \{t\geq 0; ||\nabla I(z(t+s)||\geq \varepsilon\}
\end{equation}
where $\varepsilon$ is as in Lemma (\ref{lemmaestimate}). Therefore, we have
$$b-a\geq \int_{-\infty}^{+\infty} ||\nabla I(z(t))||^{2}dt$$
$$\geq \int_{s}^{s+\tau(s)}||\nabla I(z(t))||^{2}dt$$
$$ \geq \varepsilon^{2}\tau(s)$$
therefore
$$\tau(s)\leq \frac{b-a}{\varepsilon^{2}}$$
Now
$$\lambda(s)=\lambda(s+\tau(s))-\int_{s}^{s+\tau(s)}\lambda'(s)dt$$
But notice that by construction
$$||\nabla I(z(s+\tau(s))||\leq \varepsilon$$
and $$a\leq I(z(t)) \leq b$$
Hence from Lemma (\ref{lemmaestimate}), we have that
$$ |\lambda(s+\tau(s))|\leq C$$
So we get
$$|\lambda(s)|\leq C+\sqrt{b-a}\sqrt{\tau(s)}\leq C+\frac{b-a}{\varepsilon}$$
In a similar way, we get the boundedness of $\|u\|_{H}$. Indeed,

$$u(s)=u(s+\tau(s))-\int_{s}^{s+\tau(s)}u'(t)dt.$$
From Lemma (\ref{lemmaestimate}), we have
$$\|u(s+\tau(s))\|_{H}\leq C,$$
therefore,
$$\|u(s)\|_{H}\leq C+\sqrt{\tau(s)}\sqrt{b-a}.$$
Thus from the estimate on $\tau(s)$, we have
$$\|u(s)\|_{H}\leq C+\frac{b-a}{\varepsilon}.$$
\end{proof}

\begin{proposition}
There exists a space $\hat{H}$ that embeds compactly in $H$ such that any gradient flow-line is bounded in
$$\mathcal{\hat{H}}= \hat{H}\times \mathbb{R}$$
\end{proposition}
\begin{proof}
We recall again that the space $H$ is characterized by
$$\|u\|_{H}^{2} =  \sum_{i\in \mathbb{Z}} |c_{i}|^{2} |\lambda_{i}|<\infty$$
for any $u\in E$. In a similar way, one can define $\hat{H}$ to be the set of vectors $u\in E$ such that
$$\|u\|_{\hat{H}}^{2} = \sum_{i\in \mathbb{Z}} |c_{i}|^{2} |\lambda_{i}|^{2}<\infty$$
Now let $G$ be the fundamental solution of the operator
$$\frac{d}{dt}+P_{+}-P_{-}$$
where $P_{\pm}$ is the projection on $H^{\pm}$, then from the equation of the flow, we have that
$$u(t)=\int_{-\infty}^{+\infty} G(t-s)\lambda(s) |L|^{-1}\nabla F(u) ds$$
and since $\nabla F$ maps $H$ to $E$, we have that
$$|L|^{-1}\nabla F \in \hat{H}$$
and therefore $u\in \hat{H}$.
\end{proof}

\noindent
From the previous proposition we have that the moduli spaces are modeled on the affine spaces $$\mathcal{Q}^{1}(z_{0},z_{1})=\tilde{z}+C^{1}_0(\mathbb{R},\mathcal{\hat{H}})$$
where $\tilde{z}$ is a flow line between $z_{0}$ and $z_{1}$. We will consider the map
$$ev:\mathcal{M}(z_{0},z_{1})\longmapsto \mathcal{\hat{H}}$$
defined by $ev(z)=z(0)$. This map is onto and hence the set $\mathcal{M}(z_{0},z_{1})$ is precompact.

%%%%%%%%%%%%%%%%%%%%%%%%%%%%%%%%%%%%%%%%%%%%%%%%%%%%%%%%%%%%%%%%%%%%%%%%%%%%%%%%%%%%%%%%%%%%%%%%%%%%%%%%%%%%%%%%%%%%%%%%%%%%
%%%%%%%%%%%%%%%%%%%%%%%%%%%%%%%%%%%%%%%%%%%%%%%%%%%%%%%%%%%%%%%%%%%%%%%%%%%%%%%%%%%%%%%%%%%%%%%%%%%%%%%%%%%%%%%%%%%%%%%%%%%%
%%%%%%%%%%%%%%%%%%%%%%%%%%%%%%%%%%%%%%%%%%%%%%%%%%%%%%%%%%%%%%%%%%%%%%%%%%%%%%%%%%%%%%%%%%%%%%%%%%%%%%%%%%%%%%%%%%%%%%%%%%%%

\subsubsection{Compactification by broken trajectories}

\noindent
Here we recall that the operators
$$T_{i,i+1}:\mathcal{Q}^{1}(z_{i},z_{i+1})\longrightarrow \mathcal{Q}^{0}$$
defined by
$$T_{i,i+1}(z)=\frac{dz}{dt}+\nabla I(z)$$
are Fredholm assuming transversality and
$$\mathcal{M}(z_{i},z_{i+1})=T_{i,i+1}^{-1}(0)$$
These operators then are surjective and hence they admit a right inverse $S_{i,i+1}$. Let
$$z_{01}\in\mathcal{M}(z_{0},z_{1}), \qquad z_{12}\in\mathcal{M}(z_{1},z_{2})$$
We define the function
$$z_{02,T}(t)=\Big(1-\varphi\big(\frac{t}{T}\big)\Big) z_{01}(t+2T)+\varphi\big(\frac{t}{T}\big) z_{12}(t-2T)$$
where $\varphi$ is a non-negative function such that
$$\left\{ \begin{array}{l}
\varphi(t)=0, \qquad t<-1 \\
\\
\varphi(t)=1, \qquad t\geq 1
\end{array}
\right.$$
Now we define the operator
$$A_{T}=R^{+}_{T}\tau_{2T}S_{0,1}\tau_{-2T}R^{+}_{T}+R^{-}_{T}\tau_{2T}S_{1,2}\tau_{-2T}R^{-}_{T}$$
where $\tau$ is the translation operator defined by
$$\tau_{a}f(t)=f(t+a)$$
and $R^{\pm}$ is a pair of smooth functions satisfying
$$(R^{+}_{1})^{2}+(R^{-}_{1})^{2}=1$$
$$R^{+}_{1}(t)=R^{-}_{1}(-t)$$
$$R^{\pm}_{T}=R^{\pm}(\frac{t}{T})$$
and
$$R^{+}_{1}(t)=0, \qquad t\leq -1$$
Then we have that
$$dT_{02}(z_{02,T})\circ A_{T}$$
converges to the identity operator as $T\longrightarrow \infty$. Hence by setting
$$z=z_{02,T}+A_{T}w$$
finding a connecting orbit is equivalent to solve $T_{02}(z)=0$ which can be done using the implicit function theorem in the space $\mathcal{Q}^{0}$. This is of course possible because the operator $T_{02}$ is Fredholm. Notice that this construction can be done transversally to the kernel of the linearized operator by setting
$$z=z_{02,T}+u+A_{T}w$$
where $u$ is an element in the kernel and $w$ is small.\\
In this way we have defined the gluing map by
$$z_{01}\sharp_{T,v}z_{12}=z_{02,T}+u+A_{T}w.$$
Now we can deduce that in fact the set $\mathbb{M}^{a,b}(z_{0},z_{1})$ has compact closure.

%%%%%%%%%%%%%%%%%%%%%%%%%%%%%%%%%%%%%%%%%%%%%%%%%%%%%%%%%%%%%%%%%%%%%%%%%%%%%%%%%%%%%%%%%%%%%%%%%%%%%%%%%%%%%%%%%%%%%%%%%%%%
%%%%%%%%%%%%%%%%%%%%%%%%%%%%%%%%%%%%%%%%%%%%%%%%%%%%%%%%%%%%%%%%%%%%%%%%%%%%%%%%%%%%%%%%%%%%%%%%%%%%%%%%%%%%%%%%%%%%%%%%%%%%
%%%%%%%%%%%%%%%%%%%%%%%%%%%%%%%%%%%%%%%%%%%%%%%%%%%%%%%%%%%%%%%%%%%%%%%%%%%%%%%%%%%%%%%%%%%%%%%%%%%%%%%%%%%%%%%%%%%%%%%%%%%%

\subsection{Construction of the Homology}

\noindent
In this section we will define the different chain complexes and their homologies. We will give an explicit computation later on, under specific assumptions.\\
Let $F$ be a function satisfying $(F1)$ and $(F2)$, and let $I$ denote the related energy functional. For $a<b$ we define the critical sets
$$Crit_{k}^{[a,b]}(I)$$
as the set of critical points of $I$ with energy in the interval $[a,b]$ and relative index $k$.\\
We notice that if $I$ is Morse and satisfies (PS) (which we can always assume as we will see later on), then $Crit_{k}^{[a,b]}(I)$ is always finite. Now we define the chain complex $C_{k}^{[a,b]}(I)$ as the vector space over $\mathbb{Z}_{2}$ generated by $Crit_{k}^{[a,b]}(I)$, for every $k\in \mathbb{Z}$. That is $$C_{k}^{[a,b]}(I)=Crit_{k}^{[a,b]}(I)\otimes \mathbb{Z}_{2}$$
The boundary operator $\partial$ is defined for any $z \in Crit_{k}^{[a,b]}(I)$ by
$$\partial z =\sum_{y\in Crit_{k-1}^{[a,b]}(I)} \Big(\sharp  \mathcal{M}(z,y) mod [2]\Big)y$$
Using the compactness results of the previous subsections we have that
$$\partial^{2}=0$$
and therefore $(C_{*}^{[a,b]}(I),\partial)$ is indeed a chain complex and we will denote it by $$H_{*}^{[a,b]}(I)=H_{*}(C_{*}^{[a,b]}(I),\partial)$$
its homology.

%%%%%%%%%%%%%%%%%%%%%%%%%%%%%%%%%%%%%%%%%%%%%%%%%%%%%%%%%%%%%%%%%%%%%%%%%%%%%%%%%%%%%%%%%%%%%%%%%%%%%%%%%%%%%%%%%%%%%%%%%%%%
%%%%%%%%%%%%%%%%%%%%%%%%%%%%%%%%%%%%%%%%%%%%%%%%%%%%%%%%%%%%%%%%%%%%%%%%%%%%%%%%%%%%%%%%%%%%%%%%%%%%%%%%%%%%%%%%%%%%%%%%%%%%
%%%%%%%%%%%%%%%%%%%%%%%%%%%%%%%%%%%%%%%%%%%%%%%%%%%%%%%%%%%%%%%%%%%%%%%%%%%%%%%%%%%%%%%%%%%%%%%%%%%%%%%%%%%%%%%%%%%%%%%%%%%%

\subsubsection{The $S^1$-equivariant case}

\noindent
Here we will define the equivariant homology for a particular group action, namely the $S^{1}$ action. Hence in this case, we assume $F$ to be $S^{1}$-invariant, that is
$$F(e^{i\theta}u)=F(u), \qquad \theta \in \mathbb{R}$$
We define the critical groups by
$$C_{k}^{[a,b],S^{1}}(I)=\frac{Crit_{k}^{[a,b]}(I)}{S^{1}}\otimes \mathbb{Z}_{2}$$
We notice that this definition makes sense since here the critical points are in fact critical circles, since $I$ is equivariant. Now, by breaking the symmetry perturbing the functional, each
$$z_{k}\in \frac{Crit_{k}^{[a,b]}(I)}{S^{1}}$$
splits into a max and a min respectively
$$z_{k}^{+}\in Crit_{k+1}^{[a,b]}(\tilde{I}), \qquad z_{k}^{-}\in Crit_{k}^{[a,b]}(\tilde{I})$$
where $\tilde{I}$ is the perturbed functional. Hence we define for any
$$z_{k+1}\in \frac{Crit_{k+1}^{[a,b]}(I)}{S^{1}}$$
the boundary operator
$$\partial_{S^{1}} z_{k+1}=\sum_{z_{k}\in \frac{Crit_{k}^{[a,b]}(I)}{S^{1}}} \sharp \Big( \mathcal{M}(z_{k+1}^{+},z_{k}^{+}) mod[2]\Big)z_{k}$$
We see that $\partial_{S^{1}}$ is well defined; now we need to show that indeed it is a boundary operator.
\begin{lemma}
We have
$$\partial_{S^{1}}^{2}=0$$
\end{lemma}
\begin{proof}
First we define the following chain complex
$$\overline{C}_{k}=\bigoplus_{z_{k}\in \frac{Crit_{k}^{[a,b]}(F_{B})}{S^{1}}} (z_{k}^{+},z_{k}^{-})\otimes \mathbb{Z}_{2}$$
with the following boundary operator
$$\overline{\partial} (z_{k+1}^{+},z_{k+1}^{-}) =\sum_{z_{k}\in \frac{Crit_{k}^{[a,b]}(F_{B})}{S^{1}}} (<z_{k+1}^{+},z_{k}^{+}>z_{k}^{+},<z_{k+1}^{-},z_{k}^{-}>z_{k}^{-})$$
where we put
$$<x,y>=\sharp \Big( \mathcal{M}(x,y) mod[2]\Big)$$
We claim that $\overline{\partial}^{2}=0$. Indeed this follows from the computation of
$$\partial \mathcal{M}(z_{k+1}^{+},z_{k-1}^{+})$$
This later boundary contains two kind of terms: the ones of the form
$$\mathcal{M}(z_{k+1}^{+},\overline{z}_{k+1}^{-}), \qquad \mathcal{M}(\overline{z}_{k+1}^{-},z_{k-1}^{+})$$
and those of the form
$$\mathcal{M}(z_{k+1}^{+},z_{k}^{+}), \qquad \mathcal{M}(z_{k}^{+},z_{k-1}^{+})$$
The second terms are the ones that appear in the formula for $\overline{\partial}^{2}$. So it is enough to show that the terms of the first kind cancel. In order to do that, we notice that
$$\sharp  \mathcal{M}(z_{k+1}^{+},z_{k+1}^{-}) mod [2] =0$$
and if
$$\sharp  \mathcal{M}(z_{k+1}^{+},\overline{z}_{k+1}^{-}) \not =0$$
then by the $S^{1}$ action we have that
$$\sharp  \mathcal{M}(z_{k+1}^{+},\overline{z}_{k+1}^{+}) \not =0$$
which is impossible by transversality. Hence
$$\sharp  \mathcal{M}(z_{k+1}^{+},\overline{z}_{k+1}^{-}) =0$$
and this finishes the proof of the claim that $(\overline{C}_{*},\overline{\partial})$ is a chain complex.\\
In fact this cancelation caused by the $S^{1}$ action is exactly the same as for the $\Delta$ operator in the loop space introduced by Denis and Sullivan in \cite{CS}: this operator satisfies $\Delta^{2}=0$ for the same reason as in our case.\\
Next, we consider the map
$$f_{*}:\overline{C}_{*}\longrightarrow C_{*}^{[a,b],S^{1}}(I)$$
defined by
$$f_{*}((z_{k}^{+},z_{k}^{-})=z_{k}$$
We notice that $f$ is well defined and it is an isomorphism. By the $S^{1}$ action we have that $$<z_{k+1}^{+},z_{k}^{+}>=<z_{k+1}^{-},z_{k}^{-}>$$
and we obtain that
$$\partial_{S^{1}}=f^{-1}\circ \overline{\partial} \circ f$$
Which completes the proof of the lemma.
\end{proof}

%%%%%%%%%%%%%%%%%%%%%%%%%%%%%%%%%%%%%%%%%%%%%%%%%%%%%%%%%%%%%%%%%%%%%%%%%%%%%%%%%%%%%%%%%%%%%%%%%%%%%%%%%%%%%%%%%%%%%%%%%%%%
%%%%%%%%%%%%%%%%%%%%%%%%%%%%%%%%%%%%%%%%%%%%%%%%%%%%%%%%%%%%%%%%%%%%%%%%%%%%%%%%%%%%%%%%%%%%%%%%%%%%%%%%%%%%%%%%%%%%%%%%%%%%
%%%%%%%%%%%%%%%%%%%%%%%%%%%%%%%%%%%%%%%%%%%%%%%%%%%%%%%%%%%%%%%%%%%%%%%%%%%%%%%%%%%%%%%%%%%%%%%%%%%%%%%%%%%%%%%%%%%%%%%%%%%%

\subsubsection{The $\mathbb{Z}_2$-equivariant case}

\noindent
Here we will examine the case in which $F$ is invariant by the $\mathbb{Z}_2$ group action, namely:
$$F(+u)=F(-u)$$
Therefore the related energy $I$ will be even and the critical points in this case will be pairs of the type
$$z_{k}=(\lambda_{k},u_{k}), \qquad \overline{z}_{k}=(\lambda_{k},-u_{k})$$
So we define the chain complex
$$C_{k}^{[a,b],\mathbb{Z}_{2}}(I)=\frac{Crit_{k}^{[a,b]}(I)}{\mathbb{Z}_{2}}\otimes \mathbb{Z}_{2},$$
and the related boundary operator $\partial_{\mathbb{Z}_{2}}$ defined for
$$z_{k+1}\in \frac{Crit_{k+1}^{[a,b]}(I)}{\mathbb{Z}_{2}}$$
by
$$\partial_{\mathbb{Z}_{2}}z_{k+1}=\sum_{z_{k}\in \frac{Crit_{k+1}^{[a,b]}(I)}{\mathbb{Z}_{2}}} (<z_{k+1},z_{k}>+<z_{k+1},\overline{z}_{k}>)z_{k}$$
Since the quotient map
$$f_{*}:Crit_{*}^{[a,b]}(I)\longrightarrow\frac{Crit_{*}^{[a,b]}(I)}{\mathbb{Z}_{2}}$$
extends to a map from $C_{*}^{[a,b]}(I)$ onto $C_{*}^{[a,b],\mathbb{Z}_{2}}(I)$, then the following diagram commutes:

\begin{center}
\begin{tikzpicture}
\matrix (m)[matrix of math nodes, row sep=2.6em, column sep=2.8em, text height=1.5ex, text depth=0.25ex]
{C_{k+1}^{[a,b]}(I) & C_{k}^{[a,b]}(I)\\
C_{k+1}^{[a,b],\mathbb{Z}_{2}}(I) & C_{k}^{[a,b],\mathbb{Z}_{2}}(I)\\};
\path[->,font=\scriptsize,>=angle 90]
(m-1-1) edge node[auto] {$\partial$} (m-1-2)
edge node[auto] {$f_{k+1}$} (m-2-1)
(m-1-2) edge node[auto] {$f_{k}$} (m-2-2)
(m-2-1) edge node[auto] {$\partial_{\mathbb{Z}_2}$} (m-2-2);
\end{tikzpicture}
\end{center}

\noindent
Therefore we have defined indeed a chain complex and we will write the corresponding homology as
$$H_{*}^{[a,b],\mathbb{Z}_{2}}(I)=H_{*}(C_{*}^{[a,b],\mathbb{Z}_{2}}(I),\partial_{\mathbb{Z}_{2}})$$

%%%%%%%%%%%%%%%%%%%%%%%%%%%%%%%%%%%%%%%%%%%%%%%%%%%%%%%%%%%%%%%%%%%%%%%%%%%%%%%%%%%%%%%%%%%%%%%%%%%%%%%%%%%%%%%%%%%%%%%%%%%%
%%%%%%%%%%%%%%%%%%%%%%%%%%%%%%%%%%%%%%%%%%%%%%%%%%%%%%%%%%%%%%%%%%%%%%%%%%%%%%%%%%%%%%%%%%%%%%%%%%%%%%%%%%%%%%%%%%%%%%%%%%%%
%%%%%%%%%%%%%%%%%%%%%%%%%%%%%%%%%%%%%%%%%%%%%%%%%%%%%%%%%%%%%%%%%%%%%%%%%%%%%%%%%%%%%%%%%%%%%%%%%%%%%%%%%%%%%%%%%%%%%%%%%%%%

\section{Stability}

\noindent
In this section we will consider two functions $F_{1}$, $F_{2}$ and we will show that under suitable conditions $$H_{*}(I_1)=H_{*}(I_2)$$
where we called $I_i$ the energy functional related to $F_i$, with $i=1,2$. The proof will be done in the general case and there is absolutely no difference in the equivariant case since all the perturbations can be taken to be equivariant. So, let $\eta$ be a smooth function on $\mathbb{R}$ such that
$$\left\{ \begin{array}{l}
\eta(t)=1, \qquad t\geq 1 \\
\\
\eta(t)=0, \qquad t\leq 0
\end{array}
\right.$$
We set
$$F_{t}=(1-\eta(t))F_{1}+\eta(t)F_{2}$$
and for a fixed $t\in \mathbb{R}$, we will denote by $I_t$ the energy functional related to $F_t$. Now we define the non-autonomous gradient flow by
$$z'(t)=-\nabla I_{t}(z(t)),$$
where $\nabla I_{t}$ is the gradient with respect to $z$ for a fixed $t$. Given $z_{1}$ a critical point of $I_{1}$ and $z_{2}$ a critical point of $I_{2}$, we let $z(t)$ be the flow line from $z_{1}$ to $z_{2}$.
\begin{lemma}
There exists $\delta>0$ such that if
$$|F_{1}-F_{2}|\leq \delta$$
then $z(t)$ is uniformly bounded by a constant depending only on $z_{1}$ and $z_{2}$.
\end{lemma}
\begin{proof}
Here again one needs to worry about the boundedness of $\lambda$ along the flow. First we notice that
$$\frac{\partial I_{t}(z(t))}{\partial t}=-\|z'(t)\|^{2}+\lambda \eta'(t) (F_{1}-F_{2})$$
Therefore, we have
$$I_{t}(z(t))\leq I_{1}(z_{1}) + \delta \int_{0}^{t} \lambda \eta'(s) ds$$
and
$$\int_{-\infty}^{+\infty} \|z'(t)\|^{2}dt \leq I_{2}(z_{2})-I_{1}(z_{1})+ C\delta\int_{0}^{1}|\lambda(t)|dt.$$
Now we define the same function $\tau$ as in (\ref{tau}): we need a bound for this last one. We have
$$\int_{-\infty}^{+\infty} \|\nabla I_{t}(z)(t)\|^{2}dt\leq I_{2}(z_{2})-I_{1}(z_{1})+ C\delta\int_{0}^{1}|\lambda(t)|dt$$
hence
$$\int_{s}^{s+\tau(s)} \|\nabla I_{t}(z)(t)\|^{2}dt \leq I_{2}(z_{2})-I_{1}(z_{1})+ C\delta\int_{0}^{1}|\lambda(t)|dt$$
Thus
$$\varepsilon^{2} \tau(s)\leq I_{2}(z_{2})-I_{1}(z_{1})+C\delta \|\lambda\|_{\infty}$$
Now
$$\lambda(s)=\lambda(s+\tau(s))-\int_{s}^{s+\tau(s)}\lambda'(t)dt$$
and
$$|\lambda(s)| \leq C+\sqrt{\tau(s)}\Big(I_{2}(z_{2})-I_{1}(z_{1})+C\delta \|\lambda\|_{\infty}\Big)^{\frac{1}{2}}$$
This leads to
$$|\lambda(s)|\leq C+\frac{I_{2}(z_{2})-I_{1}(z_{1})+C\delta \|\lambda\|_{\infty}}{\varepsilon}$$
Therefore, for $\delta$ chosen such that $C\delta<\varepsilon$, we have the uniform bound for $\lambda$.\\
Similarly $\|u\|_{H}$ is uniformly bounded. In fact,
$$\|u(s)\|_{H}\leq \|u(s+\tau(s))\|_{H}+\frac{I_{2}(z_{2})-I_{1}(z_{1})+C\delta \|\lambda\|_{\infty}}{\varepsilon}$$
and since $\|u(s+\tau(s))\|_{H}\leq C$ and $\lambda$ is bounded, we have the desired result.
\end{proof}

\noindent
Now, as in the autonomous case, this uniform boundedness implies pre-compactness, therefore we can define the moduli space of trajectories of the non-autonomous gradient flow,
$$\mathbb{M}(z_{1},z_{2})$$
and we omit the similar gluing construction that can be done to compactify it. In fact, we can show that it is a finite dimensional manifold with
$$dim\Big(\mathbb{M}(z_{1},z_{2})\Big)=i_{rel}(z_{1})-i_{rel}(z_{2})$$
Moreover, if
$$i_{rel}(z_{1})-i_{rel}(z_{2})=1$$
we have that
$$\partial \mathbb{M}(z_{1},z_{2})=\bigcup_{x\in Crit_{i_{rel}(z_{2})}(I_{1})} \mathbb{M}(z_{1},x)\times \mathbb{M}(x,z_{2})$$
$$\bigcup_{y\in Crit_{i_{rel}(z_{1})}(I_{2})}\mathbb{M}(z_{1},y)\times \mathbb{M}(y,z_{2})$$
With this in mind we can construct the continuation isomorphism
$$\Phi_{12} : C_{*}(I_{1})\longrightarrow C_{*}(I_{2})$$
defined at the chain level by
$$\Phi_{12}(z)=\sum_{x\in Crit_{i_{rel}(z)}(I_{2})}(\sharp \mathbb{M}(z,x) mod [2])x$$
By the previous remark on the boundary of the moduli space in the non-autonomous case, one sees that
$$\partial_{1} \Phi_{12}+\Phi_{12} \partial_{2}=0$$
this shows that it is a chain homomorphism, hence it descends at the homology level. The last thing to check is that it is an isomorphism, by taking a homotopy of homotopies (see for instance Schwarz \cite{Sh}). Finally we have the following result

\begin{corollary}
Assume that $F_{1}$ and $F_{2}$ satisfy the assumptions $(F1)$ and $(F2)$, such that $F_{1}-F_{2}$ is bounded. Then
$$H_{*}(I_{1})=H_{*}(I_{2})$$
\end{corollary}
\begin{proof}
Without loss of generality one can assume the existence of a homotopy $F_{s}$ for $s\in [0,1]$ and a partition
$$s_0=0<s_{1}<\cdots <s_{k}<1=s_{k+1}$$
such that, for a fixed $\delta>0$
$$|F_{s_{j+1}}-F_{s_{j}}|\leq \delta.$$
Hence there exist continuation isomorphisms
$$\Phi_{j,j+1}: H_{*}(I_{s_{j}})\longrightarrow H_{*}(I_{s_{j+1}}), \qquad j=0,\ldots,k$$
and since there are finitely many of them, one gets an isomorphism between $H_{*}(I_{1})$ and $H_{*}(I_{2})$.
\end{proof}

\noindent
As we said before, the same stability results hold for the equivariant cases.

%%%%%%%%%%%%%%%%%%%%%%%%%%%%%%%%%%%%%%%%%%%%%%%%%%%%%%%%%%%%%%%%%%%%%%%%%%%%%%%%%%%%%%%%%%%%%%%%%%%%%%%%%%%%%%%%%%%%%%%%%%%%
%%%%%%%%%%%%%%%%%%%%%%%%%%%%%%%%%%%%%%%%%%%%%%%%%%%%%%%%%%%%%%%%%%%%%%%%%%%%%%%%%%%%%%%%%%%%%%%%%%%%%%%%%%%%%%%%%%%%%%%%%%%%
%%%%%%%%%%%%%%%%%%%%%%%%%%%%%%%%%%%%%%%%%%%%%%%%%%%%%%%%%%%%%%%%%%%%%%%%%%%%%%%%%%%%%%%%%%%%%%%%%%%%%%%%%%%%%%%%%%%%%%%%%%%%

\section{Transversality}

\noindent
In this section we will show that up to a small and smooth perturbation of $F$ we can always assume that $I$ is Morse. Then, it can be approximated by a Morse-Smale functional with the same critical points and the same connections.

\begin{lemma}
Consider a function $F$ satisfying $(F1)$ and $(F2)$, then for a generic perturbation $K$ in $C_{0}^{3}(E)$, the energy functional $\tilde I$ related to $F+K$ is Morse.
\end{lemma}
\begin{proof}
We consider the functional
$$\psi :\mathcal{H}\times C_{0}^{3}(E)  \longrightarrow \mathcal{H}$$
defined by
$$\psi(z,K)=\nabla \tilde I (z)$$
Let us notice first that the inverse image of zero corresponds to critical points of the functional related to $F+K$. Also, for $(z,K)\in \psi^{-1}(0)$ we have
$$\partial_{z} \psi(z,K)v= Hess(\tilde I(z))v,$$
which is a perturbation of a compact operator, and hence it is a Fredholm operator of index zero. Now it remains to show that $\nabla \psi (z,K)$ is surjective. So, let us compute the differential with respect to $K$: $$\partial_{K}\psi(z,K)G=\left(\begin{array}{cc}
-\lambda |L|^{-1}\nabla G (u)\\
-G(u)
\end{array}
\right)
$$
Therefore, by taking first $G$ to be constant, we see that we can span the $\mathbb{R}$-component. For the other component we see that by taking
$$G(u)=\langle f,u \rangle$$
then we have that the range of the first component is dense since $f$ can be any function of $E$ and the operator $|L|^{-1}$ maps $E$ to a dense subspace. Thus we have the surjectivity. Therefore by the transversality theorem, $0$ is a regular point of $\psi(\cdot,K)$ for a generic $K$ and this is equivalent to say that $\tilde I$ is Morse.
\end{proof}

\noindent
Notice also that the perturbation $K$ can be chosen to be $S^{1}$- or $\mathbb{Z}_2$-equivariant if $F$ is so.

\begin{lemma}
Assume that $I$ is Morse and satisfies (PS) in $[a,b]$, then for every $\varepsilon>0$ there exists a functional $I^{\varepsilon}$ such that
\begin{itemize}
\item[(i)] $\|I-I^{\varepsilon}\|_{C^{2}}<\varepsilon$
\item[(ii)] $I^{\varepsilon}$ satisfies (PS) in $[a-\varepsilon,b+\varepsilon]$
\item[(iii)] $I^{\varepsilon}$ has the same critical points than $I$ with the same connections (number of connecting orbits).\\
\end{itemize}
\end{lemma}

\noindent
The proof of this result is similar to the one in \cite{AB2} for that it will be omitted.

%%%%%%%%%%%%%%%%%%%%%%%%%%%%%%%%%%%%%%%%%%%%%%%%%%%%%%%%%%%%%%%%%%%%%%%%%%%%%%%%%%%%%%%%%%%%%%%%%%%%%%%%%%%%%%%%%%%%%%%%%%%%
%%%%%%%%%%%%%%%%%%%%%%%%%%%%%%%%%%%%%%%%%%%%%%%%%%%%%%%%%%%%%%%%%%%%%%%%%%%%%%%%%%%%%%%%%%%%%%%%%%%%%%%%%%%%%%%%%%%%%%%%%%%%
%%%%%%%%%%%%%%%%%%%%%%%%%%%%%%%%%%%%%%%%%%%%%%%%%%%%%%%%%%%%%%%%%%%%%%%%%%%%%%%%%%%%%%%%%%%%%%%%%%%%%%%%%%%%%%%%%%%%%%%%%%%%

\section{Some remarks}

\noindent
Here we point out that in fact the most important assumption on $F$, is the hypothesis (F2). Indeed, $F$ needs to have the geometry of a local minimum around $0$. In fact, the homology that we compute is more intrinsic to the surface $S$ since the critical points will be located on the surface $S$.  In fact if $F_{1}$ and $F_{2}$ have two relative starshaped surfaces $S_{1}$ and $S_{2}$ that are bounded and if we use the notation $H_{*}(S_{i})$ instead of $H_{*}(I_{i})$ then $H_{*}(S_{1})=H_{*}(S_{2})$. This can be done using a cut off function $\eta$, that is, if the hypersurfaces $S_{i}$ is located in the ball $B_{A}(0)$ we choose $\eta$ such that
$$\eta(t)=\left\{ \begin{array}{ll}
1 \text{ for } 0\leq t\leq \sqrt{A} \\
0 \text { for } t>2\sqrt{A}
\end{array}
\right.
$$
And thus we consider the modified functional $$F_{\eta}=\eta(\|u\|^{2})F_{2}(u)+(1-\eta(\|u\|^{2}))F_{1}(u).$$ Then  $\{F_{\eta}=0\}=S_{2}$ and $F_{1}-F_{\eta}$ is bounded. Hence if we are interested in critical points in a specific surface $S$, we can disregard the behavior of the defining functional outside a bounded set and make it equal to a reference functional that we know how to compute its homology and this is the main trick in order for us to compute the homology.\\

\noindent
Regarding the meaning of the multiplier $\lambda$ instead, we will consider some examples that show up usually in the applications. First notice that if
$$0<c_{1}\leq \langle\nabla F(u),u\rangle \leq c_{2}$$
for all $u\in S$, then $\lambda$ and the energy $I$ of a critical point have the same growth, indeed if $(\lambda,u)$ is  a critical point, then
$$2\frac{|I(\lambda,u)|}{c_{2}}\leq |\lambda|\leq 2 \frac{|I(\lambda,u)|}{c_{1}}.$$
Now let us suppose that $\nabla F$ is homogeneous of degree $\alpha >0$, on $S$, then any critical point to $I(\lambda,u)$ yields a solution to the problem
\begin{equation}\label{scaling}
Lu=\nabla F(u)
\end{equation}
In fact, if $(\lambda, u)$ is a critical point of $I$, then $\lambda^{\frac{1}{\alpha-1}}u$ is a solution of (\ref{scaling}).\\
If we consider, on the other hand the problem of periodic solutions for a Hamiltonian system of the form
\begin{equation}
i\frac{\partial}{\partial t} z =\lambda \nabla H(z)
\end{equation}
Then for every $1$-periodic solutions that we get as a critical point of $I$, one has a $\lambda$-periodic solution for the system
\begin{equation}\label{per}
i\frac{\partial}{\partial t} z = \nabla H(z)
\end{equation}
Indeed if $(\lambda,z)$ is a critical point of $I$, then $z(\frac{1}{\lambda}\cdot)$ is a solution to (\ref{per}). Therefore we obtain a sequence of solutions with periods $\frac{1}{\lambda_{k}}$ and from the remark on the link between the energy and $\lambda$, we see that we have a sequence of periodic solutions with period going to infinity.

%%%%%%%%%%%%%%%%%%%%%%%%%%%%%%%%%%%%%%%%%%%%%%%%%%%%%%%%%%%%%%%%%%%%%%%%%%%%%%%%%%%%%%%%%%%%%%%%%%%%%%%%%%%%%%%%%%%%%%%%%%%%
%%%%%%%%%%%%%%%%%%%%%%%%%%%%%%%%%%%%%%%%%%%%%%%%%%%%%%%%%%%%%%%%%%%%%%%%%%%%%%%%%%%%%%%%%%%%%%%%%%%%%%%%%%%%%%%%%%%%%%%%%%%%
%%%%%%%%%%%%%%%%%%%%%%%%%%%%%%%%%%%%%%%%%%%%%%%%%%%%%%%%%%%%%%%%%%%%%%%%%%%%%%%%%%%%%%%%%%%%%%%%%%%%%%%%%%%%%%%%%%%%%%%%%%%%

\section{Computation of the Different Homologies}

\noindent
%Given $R>0$ we set $\rho_{R}$ to be a smooth function such that $\rho_{R}(s)=1$ for $s\in[0,R]$ and $\rho_{R}(s)=0$ for $s>R+1$.

%%%%%%%%%%%%%%%%%%%%%%%%%%%%%%%%%%%%%%%%%%%%%%%%%%%%%%%%%%%%%%%%%%%%%%%%%%%%%%%%%%%%%%%%%%%%%%%%%%%%%%%%%%%%%%%%%%%%%%%%%%%%
%%%%%%%%%%%%%%%%%%%%%%%%%%%%%%%%%%%%%%%%%%%%%%%%%%%%%%%%%%%%%%%%%%%%%%%%%%%%%%%%%%%%%%%%%%%%%%%%%%%%%%%%%%%%%%%%%%%%%%%%%%%%
%%%%%%%%%%%%%%%%%%%%%%%%%%%%%%%%%%%%%%%%%%%%%%%%%%%%%%%%%%%%%%%%%%%%%%%%%%%%%%%%%%%%%%%%%%%%%%%%%%%%%%%%%%%%%%%%%%%%%%%%%%%%

\subsection{The general case}

\noindent
Let us first compute the homology for the linear case, that is $H_{*}(I_0)$ where $I_0$ is the energy related to
$$F_{0}(u)=\frac{1}{2}(\|u\|^{2}-1)$$
In this case we know that the critical points of $I_0$ are the couples of eigenfunctions, eigenvalues; moreover for each eigenvalue $\lambda_{k}$ there is a circle of eigenfunctions $e^{i\theta}\psi_{k}$. By a symmetry breaking argument we can break the circle to a min and a max as we did in the construction of the homology in section 4. Therefore each critical circle of index $i_{0}$ will be broken to two critical points of index $i_{0}$ and $i_{0}+1$. Also notice that the critical circles have even index hence we have one generator for each index in the chain complex, that is $$C_{k}^{[-K,K]}=\mathbb{Z}_{2},$$ for all $k\in \mathbb{Z}$ such that $\lambda_{k}\in[-K,K]$. \\
Now let us compute the $\partial$ operator. First, by construction $\partial =0$ from odd index to even index: so it remains to compute it from even index to odd index, that is from min to max.\\
Let $(\lambda(t),u(t))$ be a flow line such that $$(\lambda(+\infty),u(+\infty))=(\lambda_{k},u_{k})$$ and $$(\lambda(-\infty),u(-\infty))=(\lambda_{k+1},u_{k+1}),$$
so if we write $u(t)$ in the basis $\varphi_{i}$ we get, if $u(t)=\sum_{i\in \mathbb{Z}}a_{i}(t)\varphi_{i}$,
$$a_{i}'=\frac{(\lambda_{i}-\lambda(t))}{|\lambda_{i}|}a_{i}.$$
Therefore we can write $a_{i}(t)=a_{i}(0)\exp(\int_{0}^{t}\frac{(\lambda_{i}-\lambda(s))}{|\lambda_{i}|}ds)$. Using the convergence of $\lambda(t)$ at infinity we get that $a_{i}=0$ for $i\not \in \{k,k+1\}$ and we have (transversally to the $S^{1}$ action) exactly one flow line from one generator to the generator of the next index. Therefore the flow is in fact a finite dimensional one. As in the finite dimensional case we get then that $\partial $ is an isomorphism from even to odd.
Therefore, one gets $H_{*}(I_0)=0$.\\

%%%%%%%%%%%%%%%%%%%%%%%%%%%%%%%%%%%%%%%%%%%%%%%%%%%%%%%%%%%%%%%%%%%%%%%%%%%%%%%%%%%%%%%%%%%%%%%%%%%%%%%%%%%%%%%%%%%%%%%%%%%%
%%%%%%%%%%%%%%%%%%%%%%%%%%%%%%%%%%%%%%%%%%%%%%%%%%%%%%%%%%%%%%%%%%%%%%%%%%%%%%%%%%%%%%%%%%%%%%%%%%%%%%%%%%%%%%%%%%%%%%%%%%%%
%%%%%%%%%%%%%%%%%%%%%%%%%%%%%%%%%%%%%%%%%%%%%%%%%%%%%%%%%%%%%%%%%%%%%%%%%%%%%%%%%%%%%%%%%%%%%%%%%%%%%%%%%%%%%%%%%%%%%%%%%%%%

\subsection{The equivariant cases}

\noindent
Let us consider again the linear case. The chain complex in this case is generated by a $\mathbb{Z}_{2}$ for each even index and the boundary operator is zero since we have a gap of 2 in the indices. Thus we have $H_{*}^{S^{1}}(I_{0})=\mathbb{Z}_{2}$ if $*$ is even and $H_{*}^{S^{1}}(I_{0})=0$ otherwise.\\
To conclude for the general case, notice that the deformations that we construct in the previous case preserves the $S^{1}$ action if we start with an $S^{1}$-equivariant $I$. Therefore $H_{*}^{S^{1}}(I)=H_{*}^{S^{1}}(I_{0})$.\\
A Similar computation in the case of the $\mathbb{Z}_{2}$ yields that $H_{*}^{\mathbb{Z}_{2}}(I)=\mathbb{Z}_{2}$.

\section{Applications}

\noindent
In this section we will present some examples of PDEs and infinite dimensional dynamical systems for which one can apply our previous results to get existence and multiplicity of solutions. We observe that we will consider examples in which the relevant operator $L$ has unbounded spectrum from above and below, which is the interesting case for our methods; however our results apply to operator such as laplacian, bilaplacian or sublaplacian as well, giving rise to the usual Morse homology: anyway, since we are looking for existence and multiplicity of solutions, we address the reader to the papers \cite{Hir}, \cite{Del2}, \cite{ding 1986}, \cite{ioalinonhbi}, \cite{ioalisuperbi}, \cite{saintier 2006}, \cite{ioalikohn}, \cite{ioaliangela}, \cite{ioalichang}, \cite{ioaligiuliochang} and the reference therein, for other kind of methods to obtain different type of existence and multiplicity results.\\

%%%%%%%%%%%%%%%%%%%%%%%%%%%%%%%%%%%%%%%%%%%%%%%%%%%%%%%%%%%%%%%%%%%%%%%%%%%%%%%%%%%%%%%%%%%%%%%%%%%%%%%%%%%%%%%%%%%%%%%%%%%%
%%%%%%%%%%%%%%%%%%%%%%%%%%%%%%%%%%%%%%%%%%%%%%%%%%%%%%%%%%%%%%%%%%%%%%%%%%%%%%%%%%%%%%%%%%%%%%%%%%%%%%%%%%%%%%%%%%%%%%%%%%%%
%%%%%%%%%%%%%%%%%%%%%%%%%%%%%%%%%%%%%%%%%%%%%%%%%%%%%%%%%%%%%%%%%%%%%%%%%%%%%%%%%%%%%%%%%%%%%%%%%%%%%%%%%%%%%%%%%%%%%%%%%%%%

\subsection{The non-linear Dirac equation}

\noindent
We start with the first application of our results. This is the case of the Dirac operator: indeed the investigation started in \cite{M} led to the present generalization. We bring the readers attention to small technical improvements we did with respect to the previously cited paper.\\
Consider a compact spin manifold $(M,g,\Sigma_{g})$, and for $1<p<\frac{n+1}{n-1}$ we propose to solve the problem
\begin{equation}\label{D}
D_{g}u=h(x)|u|^{p-1}u,
\end{equation}
for $h\geq c>0$. Then using the Rabinowitz-Floer homology, as in \cite{M}, we have that problem (\ref{D}) has an infinite sequence of solution, with energy going to infinity. Indeed this problem falls into the class of operators for which we can use the methods we introduced in this paper.\\

\noindent
We give here for the sake of completeness the different quantities we used. First, we define the functional $F:H^{\frac{1}{2}}(\Sigma M)\to \mathbb{R}$ as
$$F(u)=\frac{1}{p+1}\int_{M}\left(h(x)|u|^{p+1}(x)-1\right)dx$$
and
$$I(u,\lambda)=\frac{1}{2}\int_{M}\langle Du, u \rangle dx -\lambda F(u)$$
Since $p<\frac{n+1}{n-1}$ then assumption $(F1)$ holds, now it is easy to see that assumption $(F_{2})$ is satisfied from the structure of the nonlinearity $F$. Indeed,
$$\langle \nabla F(u),u \rangle =\int_{M} h(x)|u|^{p+1}(x) dx = 1, \text{ for } u \in S.$$
We observe that a similar result can be found in the work of T.Isobe \cite{Iso1}, where the author proves the existence of solutions using a topological linking argument.

%%%%%%%%%%%%%%%%%%%%%%%%%%%%%%%%%%%%%%%%%%%%%%%%%%%%%%%%%%%%%%%%%%%%%%%%%%%%%%%%%%%%%%%%%%%%%%%%%%%%%%%%%%%%%%%%%%%%%%%%%%%%
%%%%%%%%%%%%%%%%%%%%%%%%%%%%%%%%%%%%%%%%%%%%%%%%%%%%%%%%%%%%%%%%%%%%%%%%%%%%%%%%%%%%%%%%%%%%%%%%%%%%%%%%%%%%%%%%%%%%%%%%%%%%
%%%%%%%%%%%%%%%%%%%%%%%%%%%%%%%%%%%%%%%%%%%%%%%%%%%%%%%%%%%%%%%%%%%%%%%%%%%%%%%%%%%%%%%%%%%%%%%%%%%%%%%%%%%%%%%%%%%%%%%%%%%%

\subsection{Systems of Elliptic equations}

\noindent
Let $\Omega \subset \mathbb{R}^{n}$ be a bounded domain and let us define the following elliptic system:
\begin{equation}
\left\{ \begin{array}{lll}
-\Delta u =H_{v}(u,v) \text{ in } \Omega\\
-\Delta v=H_{u}(u,v) \text{ in } \Omega \\
u_{|\partial \Omega}=v_{|\partial \Omega}=0
\end{array}
\right.
\end{equation}
where $H$ is a smooth even function. This problem has been studied by many authors, we cite here some works and the references therein, \cite{An}, \cite{Fig}.\\

\noindent
We consider a typical non-linearity of type:
$$H(u,v)=\frac{1}{p+1}|u|^{p+1}+\frac{1}{q+1}|v|^{q+1}$$
for $\frac{1}{p+1}+\frac{1}{q+1}>\frac{n-2}{n}$.
The natural space on which we can define our functions is the space
$$\mathcal{H}=H^{1}_{0}(\Omega)\times H^{1}_{0}(\Omega) \times \mathbb{R}.$$
Then one can construct the Rabinowitz functional related to this problem in this way:
$$I(u,v,\lambda)=\int_{\Omega}\nabla u \cdot \nabla v dx-\lambda \int_{\Omega}(H(u,v)-1)dx$$
We notice here that, thinking of the first part of this paper, we have that
$$L=\left(\begin{array}{cc}
0&-\Delta \\
-\Delta & 0
\end{array}
\right)$$
and this operator have a discrete unbounded spectrum from above and below. Similarly
$$F(u,v)=\int_{\Omega} \left( \frac{1}{p+1}|u|^{p+1}+\frac{1}{q+1}|v|^{q+1} -1 \right) dx.$$
Here the level set $\{F=0\}$ bounds again a spherical domain so in particular it is strictly starshaped and thus (F2) holds; moreover from the restriction on $p$ and $q$ we have that assumption $(F1)$ holds as well. Therefore this system admits infinitely many solutions with energy going to infinity.\\

%%%%%%%%%%%%%%%%%%%%%%%%%%%%%%%%%%%%%%%%%%%%%%%%%%%%%%%%%%%%%%%%%%%%%%%%%%%%%%%%%%%%%%%%%%%%%%%%%%%%%%%%%%%%%%%%%%%%%%%%%%%%
%%%%%%%%%%%%%%%%%%%%%%%%%%%%%%%%%%%%%%%%%%%%%%%%%%%%%%%%%%%%%%%%%%%%%%%%%%%%%%%%%%%%%%%%%%%%%%%%%%%%%%%%%%%%%%%%%%%%%%%%%%%%
%%%%%%%%%%%%%%%%%%%%%%%%%%%%%%%%%%%%%%%%%%%%%%%%%%%%%%%%%%%%%%%%%%%%%%%%%%%%%%%%%%%%%%%%%%%%%%%%%%%%%%%%%%%%%%%%%%%%%%%%%%%%

\subsection{An infinite dimensional Dynamical system}

\noindent
We consider a bounded domain $\Omega \subset \mathbb{R}^{n}$ and we propose to find periodic solutions to the following infinite dimensional dynamical system:
\begin{equation}
\left\{ \begin{array}{ll}
\frac{\partial}{\partial t}u-\Delta u =H_{v}(u,v) \text{ in } \Omega\\
-\frac{\partial}{\partial t}v-\Delta v=H_{u}(u,v) \text{ in } \Omega\\
u_{|\partial \Omega}=v_{|\partial \Omega} =0
\end{array}
\right.
\end{equation}

\noindent
Again, this type of problems has been deeply investigated: we cite for example the works \cite{Bar1}, \cite{Bar2} and the references therein.\\

\noindent
Here we consider a typical non-linearity of the form
$$H(u,v)=\frac{1}{p+1}|u|^{p+1} + \frac{1}{q+1}|v|^{q+1}$$
for $1<p,q<\frac{n+2}{n}$. We have in this case, the operator
$$L=\left(\begin{array}{cc}
0&\frac{\partial}{\partial t}-\Delta \\
-\frac{\partial}{\partial t}-\Delta & 0
\end{array}
\right)$$
This is an unbounded operator on $L^{2}(S^{1}\times \Omega)$ and it is auto-adjoint with spectrum
$$\sigma(L)=\{\pm \sqrt{j^{2}+\lambda_{k}^{2}}; j\in \mathbb{Z},\lambda_{k}\in \sigma(-\Delta), k\in \mathbb{N}\}$$
The corresponding eigenfunctions, in complex notations, are of the form $$\psi_{j,k}=e^{ijt}\varphi_{k}(x)$$
where the $\varphi_{k}$ are eigenfunctions of the Laplace operator on $\Omega$. The natural space of functions to consider is then
$$H=H^{\frac{1}{2}}(S^{1},H^{1}_{0}(\Omega))$$
$$=\{u=\sum_{k,j}u_{k,j}\psi_{k,j} \in L^{2}(S^{1}\times \Omega); \sum_{k,j}|j^{2}+\lambda_{k}^{2}|^{\frac{1}{4}}u_{k,j}\psi_{k,j}\in L^{2}(S^{1}\times \Omega)\}$$
This space compactly embeds in $L^{p}$ for $1<p<\frac{n+2}{n}$. We define then
$$F(u,v)=\int_{\Omega\times S^{1}} \left( H(u,v)-1 \right)dxdt$$
In the same way as the previous example, we see that $(F1)-(F2)$ are satisfied. Therefore the system has a sequence of periodic solutions with energy going to infinity. \\
Notice that in particular, this allows us to find periodic solutions to the beam equation by taking $p=1$ in the previous function $H$.

%%%%%%%%%%%%%%%%%%%%%%%%%%%%%%%%%%%%%%%%%%%%%%%%%%%%%%%%%%%%%%%%%%%%%%%%%%%%%%%%%%%%%%%%%%%%%%%%%%%%%%%%%%%%%%%%%%%%%%%%%%%%
%%%%%%%%%%%%%%%%%%%%%%%%%%%%%%%%%%%%%%%%%%%%%%%%%%%%%%%%%%%%%%%%%%%%%%%%%%%%%%%%%%%%%%%%%%%%%%%%%%%%%%%%%%%%%%%%%%%%%%%%%%%%
%%%%%%%%%%%%%%%%%%%%%%%%%%%%%%%%%%%%%%%%%%%%%%%%%%%%%%%%%%%%%%%%%%%%%%%%%%%%%%%%%%%%%%%%%%%%%%%%%%%%%%%%%%%%%%%%%%%%%%%%%%%%

\subsection{The wave equation}

\noindent
We consider the following one dimensional wave equation:
\begin{equation}\label{L}
\left\{\begin{array}{ll}
u_{tt}-u_{xx}=f(x,u)\\
u(t+2\pi,x)=u(t,x+2\pi)=u(t,x)
\end{array}
\right.
\end{equation}
This problem was also deeply studied, we cite here only some papers: \cite{Br}, \cite{Rab2}, \cite{Tan}.\\
\noindent
We propose to find periodic solutions in time and space. The case of the wave equation presents a difference with respect to the previous ones: in fact, the linearized operator in not Fredholm since it has an infinite dimensional kernel, hence our theorem fails to apply directly. However, under some assumptions on the non-linearity we can overcome this lack of Fredholm. Again we will consider the example of a typical nonlinearity, that is for $p>1$,
$$f(u)=|u|^{p-1}u$$
First, we define the operator $L$ as follows
$$\langle Lu,v \rangle=\int_{S^{1}\times S^{1}}u_{x}v_{x}-u_{t}v_{t} dtdx$$
Then $L$ is an unbounded auto-adjoint operator on $L^{2}(S^{1}\times S^{1})$. The spectrum of $L$ is
$$\sigma(L)=\{j^{2}-k^{2};j,k \in \mathbb{N}\},$$
with eigenfunctions in complex notation of the form
$$\varphi_{i,j}(t,x)=e^{ikt}e^{ijx}.$$
The natural functional space to be considered is then $H^{1}(S^{1}\times S^{1})$. The kernel of $L$, denoted by $N$ is
$$N=\overline{span\{\varphi_{j,j}\}_{j\in \mathbb{N}}},$$
We then adopt the usual splitting of $H^{1}(S^{1}\times S^{1})$ as follows:
$$H^{1}(S^{1}\times S^{1})=H^{+}\oplus N \oplus H^{-}$$
Notice here that the kernel $N$ is an infinite dimensional space, which makes the operator $L$ not Fredholm. In order to overcome this difficulty one additional step is needed.\\
We first consider, for $u\in H=H^{+}\oplus H^{-}$, the functional $K_{u} :N\to \mathbb{R}$ defined by
$$K_{u}(v)=\frac{1}{p+1}\|u+v\|_{p+1}^{p+1}.$$
One can show that for every $u$, $K_{u}$ achieves its minimum in a function $v(u)\in N$ (see \cite{Br}, \cite{Rab2}, \cite{Tan}). We define then $Q: H \to \mathbb{R}$ by
$$Q(u)=K_{u}(v(u))=\frac{1}{p+1}\|u+v(u)\|^{p+1}_{p+1}.$$
It holds that $Q$ is compact, $C^{1}$ and since $v(u)$ is a minimum for $K_{u}$ we have that
$$\langle Q'(u),h \rangle=\int_{S^{1}\times S^{1}}|u+v(u)|^{p-1}(u+v(u))h dtdx$$
for all $h\in H$. On the other hand, using again the fact that $v(u)$ is a minimum for $K_{u}$, we get
$$\int_{S^{1}\times S^{1}}|u+v(u)|^{p-1}(u+v(u))h dtdx=0$$
for all $h\in N$. In particular
$$\int_{S^{1}\times S^{1}}|u+v(u)|^{p-1}(u+v(u))v(u) dtdx=0.$$
Now, we are able to define our Rabinowitz functional $I:\mathcal{H}=H\times \mathbb{R}\to \mathbb{R}$ by
$$I(u,\lambda)=\langle Lu,u \rangle-\lambda \left(Q(u)-1\right)$$
We are then in the setting of our paper. Since the non-linearity is polynomial then $(F1)$ holds and then we only need to verify the assumption $(F2)$. So,
$$\langle Q'(u),u \rangle=\int_{S^{1}\times S^{1}} |u+v(u)|^{p-1}(u+v(u))udxdt=$$
$$=\int_{S^{1}\times S^{1}} |u+v(u)|^{p+1}dxdt- \int_{S^{1}\times S^{1}} |u+v(u)|^{p-1}(u+v(u))v(u)dxdt$$
But, as we pointed out earlier we have
$$\int_{S^{1}\times S^{1}} |u+v(u)|^{p-1}(u+v(u))v(u)dxdt=0$$
Hence, $\langle Q'(u),u \rangle \geq c>0$ for $u\in S$ and (F2) holds true. Therefore $I$ has an infinite sequence of critical points. It is then straightforward to see that critical points of $I$ are indeed solutions of problem (\ref{L}).\\

\noindent
The same procedure can be carried out for the non-linear Schr\"{o}dinger equation of the form
\begin{equation}
\left\{\begin{array}{ll}
-i\frac{\partial }{\partial t}u=- u_{xx} +|u|^{p-1}u\\
u(t+2\pi,x)=u(t,x+2\pi)=u(x,t)
\end{array}
\right.
\end{equation}
to get the existence of an infinite sequence of solutions.\\

%Where here $u$ is periodic in time. This problem presents a similar feature to the wave equation, indeed if we consider the operato $L=-i\frac{\partial }{\partial t}+\frac{\partial^{2} }{\partial t^{2}}$ then its spectrum
%$$\sigma(L)=\{k-j^{2};k\in \mathbb{Z}, j\in \mathbb{N}\}$$
%with eigenvalues in complex notations $\varphi_{k,j}=e^{ikt}e^{ijx}$. Hence it is natural to consider the space $H$ defined by
%$$H=\{u\in L^{2}(T^{2}); \sum_{k,j}|k-j^{2}|^{\frac{1}{2}}|u_{k,j}|^{2} <\infty\}$$
%where $u_{k,j}$ are the Fourier coefficients of $u$. Again by splitting $H=H^{+}\oplus N \oplus H^{-}$ where $N$ is the kernel of $L$, one can adapt the same procedure done for the of the wave equation to get infinitely many solutions.

%%%%%%%%%%%%%%%%%%%%%%%%%%%%%%%%%%%%%%%%%%%%%%%%%%%%%%%%%%%%%%%%%%%%%%%%%%%%%%%%%%%%%%%%%%%%%%%%%%%%%%%%%%%%%%%%%%%%%%%%%%%%
%%%%%%%%%%%%%%%%%%%%%%%%%%%%%%%%%%%%%%%%%%%%%%%%%%%%%%%%%%%%%%%%%%%%%%%%%%%%%%%%%%%%%%%%%%%%%%%%%%%%%%%%%%%%%%%%%%%%%%%%%%%%
%%%%%%%%%%%%%%%%%%%%%%%%%%%%%%%%%%%%%%%%%%%%%%%%%%%%%%%%%%%%%%%%%%%%%%%%%%%%%%%%%%%%%%%%%%%%%%%%%%%%%%%%%%%%%%%%%%%%%%%%%%%%

\end{document}